\documentclass[11pt,a4paper]{amsart}
\usepackage{amsmath,amsthm,amssymb}
\usepackage{latexsym}

\allowdisplaybreaks
\numberwithin{equation}{section}

\newtheorem{theorem}{Theorem}[section]
\newtheorem*{theoremm}{Theorem}
\newtheorem{lemma}{Lemma}[section]
\newtheorem{corollary}{Corollary}[section]
\newtheorem*{corollaryy}{Corollary}
\newtheorem{proposition}{Proposition}[section]
\newtheorem{definition}{Definition}[section]

\def\cal{\mathcal}
\let\Re=\undefined
\DeclareMathOperator{\Re}{Re}
\let\Im=\undefined
\DeclareMathOperator{\Im}{Im}

\newcommand{\nn}{\nonumber}
\newcommand{\nt}{\noindent}

\newcommand{\bsl}{\backslash}

\newcommand{\pt}{\partial}
\newcommand{\ti}{\tilde}

\newcommand{\lt}{\left}
\newcommand{\rt}{\right}

\newcommand{\dsp}{\displaystyle}

\DeclareMathOperator{\diam}{\rm diam}

\DeclareMathOperator{\ppr}{Pr}

\newcommand{\br}{{\mathbb{R}}}

\newcommand{\bz}{{\mathbb{Z}}}
\newcommand{\bc}{{\mathbb{C}}}

\newcommand{\bi}{{\mathbf{i}}}
\newcommand{\bii}{{\mathbf{I}}}

\renewcommand{\nn}{{\nabla}}

\renewcommand{\l}{\lambda}

\newcommand{\vp}{\varphi}

\newcommand{\la}{\langle}
\newcommand{\ra}{\rangle}

\newcommand{\cp}{\mathcal{P}}
\newcommand{\ca}{\mathcal{A}}
\newcommand{\ccc}{\mathcal{C}}
\newcommand{\cw}{\mathcal{W}}

\begin{document}

\title[It\^o's diffusion in multidimensional scattering \ldots]
{It\^o's diffusion in multidimensional scattering with
sign-indefinite potentials.}

\author[Sergey A. Denisov]{Sergey A. Denisov}

\address{Mathematics Department, University of Wisconsin--Madison,   480 Lincoln Dr.,  Madison, WI 53706 USA}
\email{denissov@math.wisc.edu}

\keywords{Absolutely continuous spectrum, Schr\"odinger operator,
It\^o stochastic calculus, Feynman-Kac type formulae.}

\subjclass{Primary: 35P25, Secondary: 31C15, 60J45.}


\begin{abstract}
This paper extends some results of \cite{dk} to the case of
sign-indefinte potentials by applying methods developed in
\cite{d1}. This enables us to prove the presence of a.c. spectrum
for the generic coupling constant.
\end{abstract}

\maketitle

\section*{Introduction}\label{s0}

In this paper, we consider the Schr\"odinger operator
\begin{equation}
H_\lambda=-\Delta+\lambda V, \quad x\in \mathbb{R}^3 \label{s1}
\end{equation}
where $V$ is real-valued potential and $\lambda$ is a real parameter
usually called a coupling constant. We will  study the dependence of
the absolutely continuous spectrum of $H$ on the behavior of
potential $V$ by blending the methods of two papers \cite{dk} and
\cite{d1}. This question attracted much attention recently which
resulted in many publications (see, e.g. \cite{dk}-\cite{d4},
\cite{s1}-\cite{s2}, \cite{s3}-\cite{s4}) due to new very fruitful
ideas from approximation theory finding the way in the
multidimensional scattering problems.

In \cite{dk}, we introduced the stochastic differential equation
that has random trajectories as its solutions. These trajectories
are natural for describing the scattering properties of (\ref{s1})
provided that $\lambda V\geq 0$ . The reason why this requirement
was made is rooted in the method itself. Indeed, as $\lambda V\geq
0$, we have $\sigma(H_\lambda)\subseteq [0,\infty)$ for the spectrum
of $H_\lambda$. Moreover, the Green's function
$L(x,y,k)=(-\Delta+\lambda V-k^2)^{-1}(x,y,k)$ can be analytically
continued in $k$ to the whole upper half-plane $\mathbb{C}^+$ and
methods of the complex function theory can be used then. If $V$ is
sign-indefinite, the negative spectrum might occur, which is hard to
control, and this approach breaks down. The paper \cite{d1} (see
also \cite{s4}) however develops new technique which allows to
overcome this difficulty by complexifying the coupling constant and
considering the hyperbolic pencil
\[
P_\lambda(k)=-\Delta+k\lambda V-k^2
\]
instead. Then, provided that the spacial asymptotics for the new
Green's kernel $P^{-1}_\lambda(k)(x,y,k)$ is established, we can
conclude that the a.c. spectrum of $H_\lambda$ contains $[0,\infty)$
for a.e. $\lambda$. In the next section, we state the main result of
\cite{dk} and explain how it can be generalized to the
sign-indefinite case.

We are going to use the following notation. Let $\omega_R(r)$ be
infinitely smooth function on $\mathbb{R}^+$ such that
$\omega_R(r)=1$ for $r<R-1$, $\omega_R(r)=0$ for $r>R+1$, and $0\leq
\omega_R(r)\leq 1$. The function
\[
L^0(x,y,k)=\frac{e^{ik|x-y|}}{4\pi|x-y|}
\]
denotes the Green's function of the free 3d Schr\"odinger operator,
i.e. the kernel of $R_0(k^2)=(-\Delta-k^2)^{-1}$ when $\Im k>0$. The
standard symbol $B_t$ stands for the $3$-dimensional Brownian motion
and $\mathbb{S}^2$ denotes the two-dimensional unit sphere.

We consider the three-dimensional case only as it makes the writing
easy. The method however can be applied for any $d>1$. We will often
suppress the dependence on $\lambda$ unless we want to emphasize it.
\section{Main result}

 We start with stating some results from \cite{dk}. Consider the Lipschitz vector field
\[
 p(x)=\left(\frac{I'_\nu(|x|)}{I_\nu(|x|)}-\nu |x|^{-1}\right)\cdot
\frac{x}{|x|}, \quad \nu=1/2
\]
where $I_\nu$ denotes the modified Bessel function \cite[Sect.
9.6]{as}. Then, fix any point $x^0\in\mathbb{R}^3$ and consider the
following stochastic process
\begin{equation}
dX_t=p(X_t)dt+dB_t, \quad X_0=x^0 \label{stochastic}
\end{equation}
with the drift given by $p$. The solution to this diffusion process
exists and all trajectories are continuous  and escape to infinity
almost surely. One of the main results in \cite{dk} states (assume
here that $\lambda=1$)

\begin{theorem}[\cite{dk}]\label{th2}
Let $V$ be any continuous nonnegative function. Assume that $f\in
L^2(\mathbb{R}^3)$ is nonnegative and has a compact support. Let
$\sigma_f$ be the spectral measure of $f$ with respect to $H_V$ and
$\sigma'_f$ be the density of its a.c. part. Then we have
\begin{equation}
\exp\left[\frac{1}{2\pi}\int_\br \frac{\log
\sigma'_f(k^2)}{1+k^2}dk\right]\ge C_f\int f(x^0)\mathbb{E}_{x^0}
\left[ \exp\left(-\frac 12\int\limits_0^\infty V(X_\tau)
 d\tau\right)\right]dx^0 \label{th21}
\end{equation}
where the constant $C_f>0$ does not depend on $V$.
\end{theorem}

The natural corollary of this theorem is a statement that the a.c.
spectrum of $H$ fills all of $\mathbb{R}^+$ provided that the
potential $V$ is summable along the trajectory $X_t$ with positive
probability (which is the same as saying that there are
``sufficiently many" paths over which the potential is summable).

We are going to prove the following

\begin{theorem}Assume that $V$ is bounded and continuous on
$\mathbb{R}^3$ and
\begin{equation}
\int_0^\infty |V(X_t)|dt<\infty\label{rest}
\end{equation}
with positive probability. Then $\mathbb{R}^+$ supports the a.c.
spectrum of $H_\lambda$ for a.e. $\lambda$.\label{tnew}
\end{theorem}
{\bf Remark.} The conditions of continuity and boundedness of
potential are assumed for simplicity only and can probably be
relaxed. The condition of $\lambda$ being generic is perhaps also
redundant but this method does not yield any result for a particular
value of $\lambda\neq 0$. Under the conditions of the theorem, the
a.c. spectrum can be larger than the positive half-line as can be
easily seen upon taking $V=-1$ on  half-space and $V=0$ on the
complement.

We need to start with some preliminary results. They will be mostly
concerned with the study of the kernel of the operator
$P_R^{-1}(k)=(-\Delta+k\lambda V_R+k^2)^{-1}$ where
$V_R(x)=V(x)\cdot \omega_R(|x|)$. The existence of $P^{-1}(k)$ for
$\Im k>0$ as a bounded operator from $L^2(\mathbb{R}^3)$ to
$L^2(\mathbb{R}^3)$ was proved in \cite{d1}. Denote the kernel of
$P_R^{-1}(x,y,k)$ by $K_R(x,y,k)$ and compare it to the free Green's
kernel $L^0(x,y,k)$ in the following way. For $k=i$, we introduce
the amplitude
\[
b_R(\theta)=\lim_{r\to\infty} \frac{K_R(0,r\theta,i)}{L^0(0,r,i)}
\]
where $\theta\in \mathbb{S}^2$. If $L_R(x,y,k)$ is the Green's
function for $(-\Delta+|V_R|-k^2)^{-1}$, then similarly
\[
a_R(\theta)=\lim_{r\to\infty}
\frac{L_R(0,r\theta,i)}{L^0(0,r,i)},\quad
a(\theta)=\lim_{R\to\infty} a_R(\theta)
\]
In \cite{dk}, the following formulas were proved
\begin{equation}
\int\limits_{\mathbb{S}^2} a(\theta)d\theta=C_1\mathbb{E}_{X_0=0}
\left[ \exp\left(-\frac12 \int\limits_0^\infty
|V(X_\tau)|d\tau\right)\right] \label{feynman-kac}
\end{equation}

For $\theta\in\mathbb{S}^2$, let
\begin{equation}
dG_t=\theta dt+dB_t \label{potok}
\end{equation}
 Then
\begin{equation}
a(\theta)=C_2\mathbb{E}_{G_0=0} \left[ \exp\left(-\frac12
\int\limits_0^\infty |V(G_\tau)|d\tau\right)\right]\label{fk}
\end{equation}

The condition (\ref{rest}) yields
\[
\int\limits_{\mathbb{S}^2} a(\theta)d\theta>0
\]
and thus $a(\theta)>0$ for $\theta\in \Omega\subseteq \mathbb{S}^2$
and $|\Omega|>0$. In particular, that means

\begin{equation}
\mathbb{E}_{G_0=0} \left[\exp\left(-\frac 12\int_0^\infty
|V(G_t)|dt\right)\right]>0 \label{cond1}
\end{equation}
for any $\theta\in \Omega$. The bound (\ref{cond1}) is exactly what
we are going to use in this paper.

The first step is to prove an analog of the formula (\ref{fk}) for
the function $b_R(\theta)$. The  lemma below holds for any $\lambda$
so we take $\lambda=1$ for the shorthand.

\begin{lemma} If $G_t$ is defined by (\ref{potok}), then
\begin{equation}
b_R(\theta)=C_2\mathbb{E}_{G_0=0} \left[ \exp\left(-\frac{i}{2}
\int\limits_0^\infty V_R(G_\tau)d\tau\right)\right]\label{fk1}
\end{equation}
for any $R>0$.
\end{lemma}
\begin{proof} We need the following
\begin{proposition}Let $\Sigma_r$ and $B_r$ denote the sphere and closed ball of radius $r$
both centered at the origin. If $V$ is continuous and real-valued in
$B_r$ and $F(x)\in C(B_r)$, then $\phi(x)$, the solution to
\[
\frac 12 \Delta \phi+\phi_{x_1}-\frac{i}{2}V\phi=-F, \quad \phi
|_{\Sigma_r}=0
\]
admits the following representation
\begin{equation}
\phi(x)=\mathbb{E}_{G_0=x} \left[\int_0^T
\exp\left(-\frac{i}{2}\int_0^t
V(X(G_\tau))d\tau\right)F(G_t)dt\right] \label{foc}
\end{equation}
where $G_t=t(1,0,0)+B_t, G_0=x$ and $T$ is the exit time.
\end{proposition}
Notice that the solution $\phi$ always exists as the boundary
problem can be reduced to inverting the operator $-\Delta+1+iV$ with
Dirichlet boundary condition. This invertibility is a simple
corollary of the spectral theory for hyperbolic pencils and it was
proved in \cite{d1} in the context of the operators on the whole
space. The existence of the expectation in the right hand side of
(\ref{foc}) is guaranteed by $V\in \mathbb{R}$ and the fact that all
trajectories $\{G_t\}$ are continuous almost surely and the exit
time distribution has a small tail.

\begin{proof}{\it (proposition 1.1)}
This proof is quite standard for negative potentials (see \cite{af},
p.145 and \cite{oks}, lemma 7.3.2 ) but we present it here for the
reader's convenience. Take any $x, |x|<r$ and consider
\[
\Xi_t=\phi(G_t)\exp\left(-\frac{i}{2}\int_0^t V(G_\tau)d\tau\right),
\quad G_0=x, \quad t<T
\]
By It\^o's calculus,
\[
d\Xi_t=\exp\left(-\frac{i}{2}\int_0^t
V(G_\tau)d\tau\right)d\phi(G_t)-\frac{i}{2}
\phi(G_t)V(G_t)\exp\left(-\frac{i}{2}\int_0^t
V(G_\tau)d\tau\right)dt
\]
as
\[
d\phi(G_t)=\left(\phi_{x_1}(G_t)+\frac{\Delta}{2}
\phi(G_t)\right)dt+\phi_{x_1}(G_t)dB_t
\]
and
\[
d\exp\left(-\frac{i}{2}\int_0^t
V(G_\tau)d\tau\right)=-\frac{i}{2}V(G_t)\exp\left(-\frac{i}{2}\int_0^t
V(G_\tau)d\tau\right)dt
 \]
 Since
  $\Xi_0=\phi(0)$, we have
\[
\mathbb{E}_x(\Xi_T)=\phi(x)+\mathbb{E}_{x}\left[ \int_0^T
-F(G_t)\exp\left(-\frac{i}{2}\int_0^t V(G_\tau)d\tau\right)dt
\right]
\]
\[
+\mathbb{E}_x \left[\int_0^T
\phi_{x_1}(G_t)\exp\left(-\frac{i}{2}\int_0^t
V(G_\tau)d\tau\right)dB_t\right]
\]
and the last term is equal to zero. As the left hand side is equal
to zero as well due to the Dirichlet boundary conditions imposed, we
have the statement of the proposition.
\end{proof}
Now the proof of the lemma repeats  the proof of the formula (2.8)
(theorem 2.1, \cite{dk}) word for word.
\end{proof}
Assume that $\theta\in \Omega$  so (\ref{cond1}) holds. Consider
truncations $V^{(\rho)}(x)=V(x)\cdot (1-\omega_\rho(|x|))$.
\begin{lemma}\label{lg}
If $\theta\in \Omega$, then
\begin{equation}
\mathbb{E}_{G_0=0}\left[\exp\left(-\frac 12\int_0^\infty
|V^{(\rho)}(G_t)|dt\right)\right]\to 1
\end{equation}
as $\rho\to\infty$.
\end{lemma}
\begin{proof}
Take any $0<R_1<R_2$ and introduce $t_1$, the random time of hitting
the sphere $|x|=R_1$ for the first time. Then, denoting by
$\tilde{G}(t)$ the solution to (\ref{potok}) with initial condition
$G_{t_1}$, we have elementary inequality
\begin{equation}
\mathbb{E}_{G_0=0}\left[\exp\left(-\frac 12\int_0^{t_1}
|V_{R_1/2}(G_t)|dt\right)\mathbb{E}_{G_{t_1}}\left[\exp\left(-\frac
12\int_{0}^\infty
|V^{(R_2)}(\tilde{G}_t)|dt\right)\right]\right]\label{odin}
\end{equation}
\[
\geq\mathbb{E}_{G_0=0}\left[\exp\left(-\frac 12\int_0^\infty
|V(G_t)|dt\right)\right]
\]
The trajectory $G_t$ is a linear drift plus 3d Browning motion
oscillation thus for fixed $R_1$ we have decoupling
\begin{equation}
\lim_{R_2\to\infty}\left(\mathbb{E}_{G_0=0}\left[\exp\left(-\frac
12\int_0^{t_1}
|V_{R_1/2}(G_t)|dt\right)\mathbb{E}_{G_{t_1}}\left[\exp\left(-\frac
12\int_{0}^\infty
|V^{(R_2)}(\tilde{G}_t)|dt\right)\right]\right]\right) \label{dva}
\end{equation}
\[
= \mathbb{E}_{G_0=0}\left[\exp\left(-\frac 12\int_0^{t_1}
|V_{R_1/2}(G_t)|dt\right)\right]\cdot \gamma
\]
with
\[
\gamma=\lim_{R_2\to\infty} \left(\mathbb{E}_{G_0=0}\left[
\exp\left(-\frac 12\int_0^\infty
|V^{(R_2)}(G_t)|dt\right)\right]\right)
\]

On the other hand,
\[
\mathbb{E}_{G_0=0}\left[\exp\left(-\frac 12\int_0^{t_1}
|V_{R_1/2}(G_t)|dt\right)\right]\to\mathbb{E}_{G_0=0}\left[\exp\left(-\frac
12\int_0^\infty |V(G_t)|dt\right)\right], \quad R_1\to\infty
\]
which along with (\ref{odin}) and (\ref{dva}) implies $\gamma=1$.
\end{proof}

{\bf Remark.}  Let $a^{(\rho)}$ denote an amplitude for the
potential $V^{(\rho)}$. The lemma then says that $a^{(\rho)}(\theta)
\to 1$ as $\rho\to\infty$ for any $\theta\in \Omega$. Notice also
that the lemma is wrong in general if the trajectory $G_t$ is
replaced by $X_t$ as can be easily seen by letting $V=1$ on the
half-space and $V=0$ on the complement.\smallskip

Now we are ready for the proof of theorem \ref{tnew}.
\begin{proof}{\it (theorem \ref{tnew})}

Notice first that the standard trace-class perturbation argument
\cite{rs} implies that
$\sigma_{ac}(-\Delta+V_\lambda)=\sigma_{ac}(-\Delta
+V_\lambda^{(\rho)})$ for any $\rho$. Fix some large $c>0$ and take
$\lambda\in [-c,c]$. Then, by lemma \ref{lg}, we can make $\rho$
large enough so that for any $\theta\in \Omega_1\subseteq \Omega$,
we have
\[
\mathbb{E}\left[\exp\left(-\frac{c}{2}\int_0^\infty
|V^{(\rho)}(G_t)|dt\right)\right]>0.99
\]
Then, due to (\ref{fk1}),
\begin{equation}
|b_R^{(\rho)}(\theta)|>1/2\label{oc1}
\end{equation}
for any $\theta\in \Omega_1$ and any $R>\rho$. Now, we need to
recall several results from \cite{d1} and repeat a couple of
arguments from this paper. Let $f=\chi_{|x|<1}$. Denote the spectral
measure of $f$ with respect to $-\Delta+\lambda V^{(\rho)}_R$ by
$\sigma_{\rho,R}(E,\lambda)$.  For $k\in \mathbb{C}^+$, consider
\[
J_{\rho,R}(k,\theta,\lambda)=\lim_{r\to\infty}
\frac{\left((-\Delta+\lambda
kV^{(\rho)}_R+k^2)^{-1}f\right)(r\theta)}{r^{-1}e^{ikr}}
\]
Then, the formula (38) from \cite{d1} says
\begin{equation}
\sigma'_{\rho,R}(k^2,k\lambda)=k\pi^{-1}\|J_{\rho,R}(k,\theta,\lambda)\|_{L^2(\mathbb{S}^2)}^2
\label{fz}
\end{equation}
where $k\neq 0$ is real and in the right hand side the limiting
value of $B_{\rho,R}$ as $\Im k\to +0$ is taken. This function
$J_{\rho,R}(k,\theta,\lambda)$ is continuous on
$\overline{\mathbb{C}^+}\backslash \{0\}$ as seen from the
absorption principle (\cite{rs4}, chapter 13, section 8 or
\cite{d1}). Around zero, we have an estimate
\begin{equation}
|J_{\rho,R}(k,\theta,\lambda)|<C(\rho,R)|k|^{-1}\label{estg}
\end{equation}
that can be deduced from the representation
\[
P^{-1}(k)f=R_0(k^2)f-k\lambda R_0(k^2)V_R^{(\rho)}P^{-1}(k)f\quad
k\in \mathbb{C}^+
\]
and an estimate
\[
\|P^{-1}(k)\|\leq (\Im k)^{-2}
\]
(see (37), \cite{d1}).

For large $|k|$, we have the following uniform estimate
\begin{equation}
\int_{\mathbb{S}^2} |J_{\rho,R}(k,\theta,\lambda)|^2 d\theta <C
\frac{1+|k|\Im k}{[\Im k]^4} \|f(x)\|_2 \|f(x)e^{2\Im
k|x|}\|_2\label{estg1}
\end{equation}
(take $r\to\infty$ in (48), \cite{d1}).

Now,  consider the function
\[
g(k)=\ln
\|ke^{2ik}J_{\rho,R}(k,\theta,\lambda)\|_{L^2(\mathbb{S}^2)}
\]
This function is subharmonic in $\mathbb{C}^+$ and the estimates
(\ref{estg}) and (\ref{estg1}) enable us to apply the mean-value
inequality with the reference point $k=i$.

This, along with the identity (\ref{fz}), gives
\[
\int_\mathbb{R} \frac{\ln \sigma'_{\rho,R}(k^2,k\lambda)}{k^2+1}dk>
C_1+C_2\ln \int_{\mathbb{S}^2}
|J_{\rho,R}(i,\theta,\lambda)|^2d\theta, \quad C_2>0
\]
Now, notice that the choice of $f$ guarantees that
\[
|J_{\rho,R}(i,\theta,\lambda)|\sim
|b_{R}^{(\rho)}(i,\theta,\lambda)|
\]
and (\ref{oc1}) implies that
\[
C_1 +C_2\ln \int_{\mathbb{S}^2}
|J_{\rho,R}(i,\theta,\lambda)|^2d\theta>C_3
\]
uniformly in $R$ and $\lambda\in [-c,c]$. Thus, we have an estimate
\[
\int_\mathbb{R} \frac{\ln
\sigma'_{\rho,R}(k^2,k\lambda)}{k^2+1}dk>-C
\]
uniformly in $R>\rho$ and $\lambda\in [-c,c]$ with any fixed $c$.
Now, taking any interval $(a,b)\subset (0,\infty)$, we have
\[
\int_a^b dE \int _{-c}^c \ln \sigma_{\rho,R}'(E,\lambda)d\lambda>-C
\]
uniformly in $R$ so taking $R\to\infty$ and using the lower
semicontinuity of the entropy (see \cite{ks} and \cite{d1}, p.21 and
Lemma 3.4), we get
\[
\int_a^b dE \int _{-c}^c \ln \sigma_{\rho}'(E,\lambda)d\lambda>-C
\]
The Fubini-Tonelli theorem now gives
\[
\int_a^b \ln \sigma_{\rho}'(E,\lambda)dE>-\infty
\]
for a.e. $\lambda$ so $[a,b]\subseteq
\sigma_{ac}(H^{(\rho)}_\lambda)$ for a.e. $\lambda$. As was
mentioned already, the a.c. spectrum is stable under changing the
potential on any compact set. Thus, $[a,b]\subseteq
\sigma_{ac}(H_\lambda)$ for a.e. $\lambda$. Since $[a,b]$ was taking
arbitrarily, we have statement of the theorem.
\end{proof}

{\bf Remark.} In the paper \cite{dk}, we studied the case when the
potential $V\geq 0$ and is supported on the set $E$ (a good example
to think about is a countable collection of  balls) of any geometric
structure. The special modified capacity and the harmonic measure
were introduced and studied  which allowed the effective estimation
of probabilities in the natural geometric terms. The same results,
e.g. the estimates in terms of anisotropic Hausdorff content and the
size of the spherical projection, are true in the current setting
when the potential is not assumed to be positive. The statements
however are true only generically in $\lambda$.

\bigskip\nt
{\bf Acknowledgement.} \rm We acknowledge the support by Alfred P.
Sloan Research Fellowship and the NSF grant DMS-0758239. Thanks go
to Stas Kupin and the University of Provence where part of this work
done.

\end{document}